\documentclass[12pt]{amsart}
\usepackage{amsmath,amsthm,amsfonts,amssymb,mathrsfs}

\usepackage{color}
\input xymatrix
\xyoption{all}

\usepackage{hyperref}

\newtheorem{theorem}{Theorem}[section]
\newtheorem{claim}{Claim}

\newtheorem{proposition}[theorem]{Proposition}
\newtheorem{corollary}[theorem]{Corollary}
\newtheorem{lemma}[theorem]{Lemma}
\theoremstyle{definition}

\def\t_c{\tau_{comp}}

\def \eps{\varepsilon}
\begin{document}

\title[On topological McAlister semigroups]
{On topological M\MakeLowercase{c}Alister semigroups}

\author{Serhii Bardyla}

\address{Universit\"at Wien,
Institut f\"ur Mathematik, Kurt G\"odel Research Center, Wien, Austria.}
\email{sbardyla@yahoo.com}

\makeatletter
\@namedef{subjclassname@2020}{%
  \textup{2020} Mathematics Subject Classification}
\makeatother


\thanks{The author was supported by
the Austrian Science Fund FWF (Grant M 2967).}

\begin{abstract}
In this paper we consider McAlister semigroups over arbitrary cardinals and investigate their algebraic and topological properties. We show that the group of automorphisms of a McAlister semigroup $\mathcal{M}_{\lambda}$ is isomorphic to the direct product $Sym(\lambda){\times}\mathbb{Z}_2$, where $Sym(\lambda)$ is the group of permutations of the cardinal $\lambda$.
This fact correlates with the result of Mashevitzky, Schein and Zhitomirski which states that the group of automorphisms of the free inverse semigroup over a cardinal $\lambda$ is isomorphic to the wreath product of $Sym(\lambda)$ and $\mathbb{Z}_2$.
Each McAlister semigroup admits a compact semigroup topology. Consequently, the Green's relations $\mathscr D$ and $\mathscr J$ coincide in McAlister semigroups. The latter fact complements results of Lawson. We showed that each non-zero element of a Hausdorff semitopological McAlister semigroup is isolated. This fact is an analogue of the result of Mesyan, Mitchell, Morayne and P\'{e}resse, who proved that each non-zero element of Hausdorff topological polycyclic monoid is isolated. Also, it follows that the free inverse semigroup over a singleton admits only the discrete Hausdorff shift-continuous topology. We proved that a Hausdorff locally compact semitopological semigroup $\mathcal{M}_1$ is either compact or discrete. This fact is similar to the result of Gutik, who showed that a Hausdorff locally compact semitopological polycyclic monoid $\mathcal{P}_1$ is either compact or discrete. However, this dichotomy does not hold for the semigroup $\mathcal{M}_2$. Moreover, $\mathcal{M}_2$ admits continuum many different Hausdorff locally compact inverse semigroup topologies.
 \end{abstract}

\maketitle

\section{Introduction and Preliminaries}
We shall follow the terminology
of~\cite{Engelking-1989, Lawson-1998}.
All topological spaces are assumed to be Hausdorff. A semigroup $S$ is called an {\em inverse semigroup} if for each element $a\in S$ there exists a unique inverse element $a^{-1}\in S$ such that $aa^{-1}a=a$ and $a^{-1}aa^{-1}=a^{-1}$.
The map $inv:S\rightarrow S$, $a\rightarrow a^{-1}$ is called an {\em inversion}.

A {\em topologized semigroup} is a topological space endowed with a semigroup operation. A topologized semigroup $S$ is called
\begin{itemize}
\item {\em semitopological}, if for each element $x\in S$ the shifts $l_x(s):s\rightarrow xs$ and $r_x(s):s\rightarrow sx$ are continuous;
\item{\em topological}, if the semigroup operation is continuous in $S$;
\item{\em topological inverse}, if $S$ is an inverse semigroup, and semigroup operation together with inversion are continuous in $S$.
\end{itemize}

A topology $\tau$ on a semigroup $S$ is called
\begin{itemize}
\item {\em shift-continuous}, if $(S,\tau)$ is a semitopological semigroup;
\item{\em semigroup}, if $(S,\tau)$ is a topological semigroup;
\item{\em inverse semigroup}, if $(S,\tau)$ is a topological inverse semigroup.
\end{itemize}


A finite sequence of elements of a nonempty set $A$ is called {\em a word over} $A$. The empty word is denoted by $\eps$.
For a non-zero cardinal $\lambda$ by $F_{\lambda}$ we denote the free monoid over a set $A$ of cardinality $\lambda$. Note that $F_{\lambda}$ can be represented as a set of all words over the set $A$ endowed with the semigroup operation of concatenation. More precisely, the concatenation of two words $a=a_1\ldots a_n\in F_{\lambda}$ and $b=b_1\ldots b_m\in F_{\lambda}$ is the word $ab=a_1\ldots a_nb_1\ldots b_m\in F_{\lambda}$. We agree that $\eps a=a\eps=a$ for each $a\in F_{\lambda}$, that is $\eps$ is the unit of $F_{\lambda}$.
By $|a|$ we denote the length of a word $a\in F_{\lambda}\setminus\{\eps\}$. We assume that $|\eps|=0$. If $a=a_1\ldots a_n\in F_{\lambda}$, then let $a^t=a_n\ldots a_1$. Also, we agree that $a^t=a$ for any $a\in F_{\lambda}$ such that $|a|\leq 1$.

For a non-zero cardinal $\lambda$ the polycyclic monoid $\mathcal{P}_\lambda$ is the set $F_{\lambda}{\times}F_{\lambda}\sqcup\{0\}$ endowed with the following semigroup operation:

\begin{equation*}
\begin{split}
  & (a,b)\cdot (c,d)=
    \left\{
      \begin{array}{ccl}
       (c_{1}a,d), & \hbox{if~~} c=c_{1}b & \hbox{for some~} c_1\in F_{\lambda};\\
        (a,b_{1}d),   & \hbox{if~~} b=b_{1}c & \hbox{for some~} b_1\in F_{\lambda};\\
        0,              & \hbox{otherwise;}   &
      \end{array}
    \right.\\
  &  \hbox{~and~} \quad (a,b)\cdot 0=0\cdot (a,b)=0\cdot 0=0.
    \end{split}
\end{equation*}
Polycyclic monoids over finite cardinals were introduced by Nivat and Perrot in~\cite{Nivat-Perrot-1970}. Basic algebraic properties of polycyclic monoids are described in Chapter~9.3 from~\cite{Lawson-1998}. In~\cite{BardGut-2016(1)} Gutik and the author investigated algebraic properties of polycyclic monoids over arbitrary cardinals. It turns out that they share many common properties with the classical polycyclic monoids. Nowadays polycyclic monoids are well-studied algebraic objects
(see~\cite{Jones-Lawson-2012,Lawson-2009}) and have fruitful applications.  In particular, they are universal objects in the class of graph inverse semigroups. Namely, each graph inverse semigroup $G(E)$ over a directed graph $E$ is a subsemigroup of the polycyclic monoid $\mathcal{P}_{|G(E)|}$ ~\cite{Bardyla-2017(2)}. Also, polycyclic monoids are useful in a construction of Thompson groups (see ~\cite{Lawson-2007, Lawson-2020}).

It is easy to see that the set $I=\mathcal{P}_{\lambda}{\times}\{0\}\cup \{0\}{\times}\mathcal{P}_{\lambda}$ is a two-sided ideal of the direct product $\mathcal{P}_{\lambda}{\times}\mathcal{P}_{\lambda}$. Let $\mathcal{S}_{\lambda}$ denotes the Rees quotient semigroup $(\mathcal{P}_{\lambda}{\times}\mathcal{P}_{\lambda})/I$. By $0$ we naturally denote the point $q(I)\in \mathcal{S}_{\lambda}$, where $q$ is the quotient map. The subsemigroup $\mathcal{M}_{\lambda}=\{((a,b),(c,d))\mid \varepsilon\neq ac^t=bd^t\}\cup\{0\}$ of $\mathcal{S}_{\lambda}$ is called the {\em McAlister semigroup}. The semigroups $\mathcal{M}_k$, $k\in\mathbb{N}$ were introduced by McAlister in~\cite{M}.
McAlister semigroups are closely related to tiling semigroups constructed from a one-dimensional tiling. Also, they can be considered as a generalization of the free inverse semigroup with one generator, because the semigroup $\mathcal{M}_1$ is isomorphic to the free inverse semigroup over a singleton with adjoined zero~\cite{L}.
For alternative representations (which do not use polycyclic monoids) and other algebraic properties of McAlister semigroups see Chapter~9.4 from~\cite{Lawson-1998} or papers of Lawson and McAlister~\cite{L,M}.

Let us note that polycyclic monoids as well as McAlister semigroups are closely related to the well-known bicyclic monoid $\mathcal{B}$. In particular, the polycyclic monoid $\mathcal{P}_1$ is isomorphic to the bicyclic monoid with adjoined zero. Topological bicyclic monoid was investigated by Eberhart and Selden~\cite{Eberhart-Selden-1969}. They showed that the bicyclic semigroup admits only the discrete semigroup topology, and if ${\mathcal{B}}$ is a dense proper subsemigroup of a topological semigroup $S$, then $I=S\setminus{\mathcal{B}}$ is a two-sided ideal of $S$. Bertman and West~\cite{Bertman-West-1976} showed that the bicyclic monoid admits only the discrete shift-continuous topology.
Compact topological semigroups cannot contain an isomorphic copy of the bicyclic monoid~\cite{Anderson-Hunter-Koch-1965}.
Embedding of the bicyclic monoid into compact-like topological semigroups was discussed
in~\cite{Banakh-Dimitrova-Gutik-2010, GutRep-2007, Hildebrant-Koch-1988}.

Topological polycyclic monoids and their subsemigroups were investigated by Mesyan, Mitchell, Morayne and P\'{e}resse~\cite{Mesyan-Mitchell-Morayne-Peresse-2013}. In particular, they showed that each non-zero element of a topological graph inverse semigroup is isolated and a locally compact topological graph inverse semigroup over a finite graph $E$ is discrete. Since a polycyclic monoid $\mathcal{P}_{\lambda}$ is isomorphic to the graph inverse semigroup over the graph $E$ which contains one vertex and $\lambda$ many loops, their results imply that each non-zero element of a topological polycyclic monoid is isolated and locally compact topological polycyclic monoids $\mathcal{P}_k$, $k\in\mathbb{N}$ are discrete. Gutik and the author in~\cite{BardGut-2016(1)} showed that each non-zero element of a semitopological polycyclic monoid is isolated and for every non-zero cardinal $\lambda$ a locally compact topological polycyclic monoid $\mathcal{P}_{\lambda}$ is discrete. Graph inverse semigroups which admit only the discrete locally compact semigroup topology were characterized in~\cite{Bardyla-2017(1)}.

In this paper we investigate algebraic and topological properties of McAlister semigroups.

\section{Algebraic properties of McAlister semigroups}

\begin{lemma}\label{l0}
For each non-zero element $x\in\mathcal{M}_{\lambda}$ there exist words $u,v,w\in F_{\lambda}$ such that $x=((u,uv),(wv^t,w))$ or $x=((uv,u),(w,wv^t))$.
\end{lemma}

\begin{proof}
Let $((a,b),(c,d))\in\mathcal{M}_{\lambda}\setminus\{0\}$, where $a=a_1\ldots a_n$, $b=b_1\ldots b_m$, $c=c_1\ldots c_k$, $d=d_1\ldots d_p$.
Recall that $ac^t=bd^t\neq \eps$ witnessing that $a_1\ldots a_nc_k\ldots c_1=b_1\ldots b_md_p\ldots d_1$.
Three cases are possible:
\begin{enumerate}
\item $n>m$ and $k<p$;
\item $n<m$ and $k>p$;
\item $n=m$ and $k=p$.
\end{enumerate}

1) It follows that $a_1=b_1,\ldots, a_m=b_m, a_{m+1}=d_p,\ldots, a_n=d_{p-n+m+1}, c_k=d_{p-n+m}, \ldots ,c_1=d_1$.
Put $v\equiv a_{m+1}\ldots a_n$, $u\equiv a_1\ldots a_m$ and $w\equiv c_1\ldots c_k$. Then $((a,b),(c,d))=((uv,u),(w,wv^t))$.

2) It follows that $b_1=a_1,\ldots, b_n=a_n, b_{n+1}=c_k,\ldots, b_m=c_{k-m+n+1}, d_p=c_{k-m+n}, \ldots, d_1=c_1$.
Put $v\equiv b_{n+1}\ldots b_m$, $u\equiv a_1\ldots a_n$ and $w\equiv d_1\ldots d_p$. Then $((a,b),(c,d))=((u,uv),(wv^t,w))$.

3) It follows that $a_i=b_i$ for any $i\leq n$ and $c_i=d_i$ for any $i\leq k$. Put $v\equiv \eps$, $u\equiv a$ and $w\equiv c$. Then $((a,b),(c,d))=((a,a),(c,c))=((a\eps,a),(c,c\eps^t))=((uv,u),(w,wv^t))$.
\end{proof}

For an inverse semigroup $S$ the Green's relations $\mathscr{L}$,
$\mathscr{R}$, $\mathscr{H}$, $\mathscr{D}$ and $\mathscr{J}$ are defined as follows:
\begin{center}
\begin{tabular}{rcl}
  $a\mathscr{L}b$ & if and only if & $a^{-1}a=b^{-1}b$; \\
  $a\mathscr{R}b$ & if and only if & $aa^{-1}=bb^{-1}$; \\
  $a\mathscr{J}b$ & if and only if & $SaS=SbS$;\\
  & $\mathscr{H}=\mathscr{L}\cap\mathscr{R}$;&\\
  & $\mathscr{D}=\mathscr{L}{\circ}\mathscr{R}=\mathscr{R}{\circ}\mathscr{L}$. &\\
\end{tabular}
\end{center}

The Green's relations $\mathscr{L},\mathscr{R},\mathscr{H}$ and $\mathscr{D}$ on McAlister semigroups $\mathcal{M}_n$, $n\in\mathbb{N}$ were described by Lawson in~\cite{L}.
The following lemma complements mentioned above results.
\begin{lemma}\label{l1}
Let $((a_1,b_1),(c_1,d_1))$ and $((a_2,b_2),(c_2,d_2))$ be arbitrary non-zero elements of a semigroup $\mathcal{M}_{\lambda}$. Then the following conditions hold:
\begin{enumerate}
\item $((a_1,b_1),(c_1,d_1))\mathscr{L}((a_2,b_2),(c_2,d_2))$ iff $b_1=b_2$ and $d_1=d_2$;
\item $((a_1,b_1),(c_1,d_1))\mathscr{R}((a_2,b_2),(c_2,d_2))$ iff $a_1=a_2$ and $c_1=c_2$;
\item $\mathscr{H}$-classes are singletons;
\item $((a_1,b_1),(c_1,d_1))\mathscr{D}((a_2,b_2),(c_2,d_2))$ iff $a_1c_1^t=a_2c_2^t$;
\item the relations $\mathscr{D}$ and $\mathscr{J}$ coincide on $\mathcal{M}_{\lambda}$.
\end{enumerate}
\end{lemma}

\begin{proof}
Note that for any $((a,b),(c,d))\in\mathcal{M}_{\lambda}\setminus\{0\}$, $((a,b),(c,d))^{-1}=((b,a),(d,c))$.

1) Observe that
$$((a_1,b_1),(c_1,d_1))^{-1}\cdot((a_1,b_1),(c_1,d_1))=((b_1,a_1),(d_1,c_1))\cdot((a_1,b_1),(c_1,d_1))=((b_1,b_1),(d_1,d_1)).$$
Similarly it can be checked that $((a_2,b_2),(c_2,d_2))^{-1}\cdot((a_2,b_2),(c_2,d_2))=((b_2,b_2),(d_2,d_2))$.
By the definition of the relation $\mathscr{L}$,
$$((a_1,b_1),(c_1,d_1))\mathscr{L}((a_2,b_2),(c_2,d_2))\qquad \hbox{iff}\qquad  ((b_1,b_1),(d_1,d_1))=((b_2,b_2),(d_2,d_2)).$$
The latter equality is true iff $b_1=b_2$ and $d_1=d_2$.
Hence condition 1 holds.

2) Observe that
$$((a_1,b_1),(c_1,d_1))\cdot((a_1,b_1),(c_1,d_1))^{-1}=((a_1,b_1),(c_1,d_1))\cdot((b_1,a_1),(d_1,c_1))=((a_1,a_1),(c_1,c_1)).$$
Similarly it can be checked that $((a_2,b_2),(c_2,d_2))\cdot((a_2,b_2),(c_2,d_2))^{-1}=((a_2,a_2),(c_2,c_2))$.
By the definition of the relation $\mathscr{R}$,
$$((a_1,b_1),(c_1,d_1))\mathscr{R}((a_2,b_2),(c_2,d_2))\qquad \hbox{iff}  \qquad((a_1,a_1),(c_1,c_1))=((a_2,a_2),(c_2,c_2)).$$
The latter equality is true iff $a_1=a_2$ and $c_1=c_2$.
Hence condition 2 holds.

3) By the definition of the relation $\mathscr H$, $((a_1,b_1),(c_1,d_1))\mathscr{H}((a_2,b_2),(c_2,d_2))$ iff $$((a_1,b_1),(c_1,d_1))\mathscr{L}((a_2,b_2),(c_2,d_2))\quad \hbox{ and } \quad((a_1,b_1),(c_1,d_1))\mathscr{R}((a_2,b_2),(c_2,d_2)).$$ At this point condition 3 follows from conditions 1 and 2.

4) ($\Rightarrow$) Since $\mathscr{D}=\mathscr{L}{\circ}\mathscr{R}$ there exists an element $((a_3,b_3),(c_3,d_3))\in \mathcal{M}_{\lambda}\setminus\{0\}$ such that
$$((a_1,b_1),(c_1,d_1))\mathscr{L}((a_3,b_3),(c_3,d_3))\quad \hbox{ and }\quad ((a_3,b_3),(c_3,d_3))\mathscr{R}((a_2,b_2),(c_2,d_2)).$$
Conditions 1 and 2 imply that $b_1=b_3$, $d_1=d_3$ and $a_2=a_3$, $c_2=c_3$. Then $$a_1c_1^t=b_1d_1^t=b_3d_3^t=a_3c_3^t=a_2c_2^t.$$

($\Leftarrow$) Since $a_2c_2^t=a_1c_1^t=b_1d_1^t$, we have that $((a_2,b_1),(c_2,d_1))\in\mathcal{M}_{\lambda}\setminus\{0\}$. By conditions 1 and 2, $$((a_1,b_1),(c_1,d_1))\mathscr{L}((a_2,b_1),(c_2,d_1))\qquad \hbox{and}\qquad ((a_2,b_1),(c_2,d_1))\mathscr{R}((a_2,b_2),(c_2,d_2)).$$ Hence $((a_1,b_1),(c_1,d_1))\mathscr{D}((a_2,b_2),(c_2,d_2))$.

5) By Proposition 2.19 from~\cite{HS}, the relations $\mathscr D$ and $\mathscr J$ coincide on each compact topological semigroup $S$. So, it is sufficient to construct a compact semigroup topology on $\mathcal{M}_{\lambda}$. It will be done in Theorem~\ref{c} after some preparatory work made before.
\end{proof}

Observe that to each $\mathscr{D}$-class $D\subset \mathcal{M}_{\lambda}\setminus\{0\}$ corresponds a certain word $p\in F_{\lambda}\setminus\{\eps\}$. Namely, fix any $((a,b),(c,d))\in D$ and put $p=ac^t$. By Lemma~\ref{l1}, the word $p$ does not depend on the choice of $((a,b),(c,d))\in D$. Further, by $D_p$ we will denote the $\mathscr D$-class in $\mathcal{M}_{\lambda}$ which contains an element $((a,b),(c,d))$ with $ac^t=p$. The following lemma is an analogue of Proposition 8 from~\cite{L} but for McAlister semigroups over arbitrary cardinals.

\begin{lemma}\label{l2}
For each $p\in F_{\lambda}\setminus\{\eps\}$ the following statements hold:
\begin{itemize}
\item $|D_p|=(|p|+1)^2$;
\item if $e\leq f$ for some idempotents $e,f\in D_p$, then $e=f$.
 \end{itemize}
\end{lemma}
\begin{proof}
Observe that there are $|p|+1$ possibilities to represent the word $p\in\mathcal{F}_{\lambda}\setminus\{\eps\}$ as $ac^t$, where $a,c\in F_{\lambda}$. It follows that there are $(|p|+1)^2$ possibilities to construct an element $((a,b),(c,d))$ such that $ac^t=bd^t=p$. Thus the set $D_p$ contains $(|p|+1)^2$  elements. Note that idempotents of $\mathcal{M}_{\lambda}\setminus\{0\}$ have the form $((a,a),(c,c))$, where $ac^t\neq\eps$. Fix any two idempotents $((a,a),(c,c))\leq((b,b),(d,d))$ in $D_p$. It follows that $((a,a),(c,c))\cdot((b,b),(d,d))=((a,a),(c,c))$. The definition of the semigroup operation in $\mathcal{M}_{\lambda}$ implies that there exists $a_1,c_1\in F_{\lambda}$ such that $a=a_1b$ and $c=c_1d$. Since $bd^t=p=ac^t=a_1bd^tc_1^t$ we deduce that $a_1=c_1=\eps$. Hence $((a,a),(c,c))=((b,b),(d,d))$.
\end{proof}

\begin{proposition}\label{p1}
Fix any words $q,p\in F_{\lambda}\setminus\{\eps\}$ and elements $((a,b),(c,d))\in D_p$, $((x,y),(u,v))\in D_q$. If ${\bf z}=((a,b),(c,d))\cdot((x,y),(u,v))\neq 0$, then there exist words $s_1,r_1,s_2,r_2\in F_{\lambda}$ such that $s_1p r_1=s_2q r_2$ and ${\bf z}\in D_{s_1p r_1}=D_{s_2q r_2}$.
\end{proposition}

\begin{proof}
Recall that $p=ac^t=bd^t$ and $q=xu^t=yv^t$. Since ${\bf z}\neq 0$ one of the following cases holds.
\begin{enumerate}
\item $x=wb$ and $u=zd$ for some $w,z\in F_{\lambda}$;
\item $x=wb$ and $d=zu$ for some $w,z\in F_{\lambda}$;
\item $b=wx$ and $d=zu$ for some $w,z\in F_{\lambda}$;
\item $b=wx$ and $u=zd$ for some $w,z\in F_{\lambda}$.
\end{enumerate}

1) Observe that
$${\bf z}=((a,b),(c,d))\cdot((wb,y),(zd,v))=((wa,y),(zc,v)).$$
It follows that $wac^tz^t=yv^t$ which implies that $wpz^t=q$. Hence ${\bf z}\in D_q=D_{wpz^t}$.
Put $s_1=w, r_1=z^t, s_2=r_2=\eps$.

2) Observe that
 $${\bf z}=((a,b),(c,zu))\cdot((wb,y),(u,v))=((wa,y),(c,zv)).$$
 It follows that $wac^t=yv^tz^t$ which implies that  $wp=qz^t$. Hence ${\bf z}\in D_{wp}=D_{qz^t}$.
Put $s_1=w, r_1=s_2=\eps, r_2=z^t$.

3) Observe that
$${\bf z}=((a,wx),(c,zu))\cdot((x,y),(u,v))=((a,wy),(c,zv)).$$
It follows that $ac^t=wyv^tz^t$ which implies that  $p=wqz^t$. Hence ${\bf z}\in D_{p}=D_{wqz^t}$.
Put $s_1=r_1=\eps, s_2=w, r_2=z^t$.

4)Observe that $${\bf z}=((a,wx),(c,d))\cdot((x,y),(zd,v))=((a,wy),(zc,v)).$$
It follows that $ac^tz^t=wyv^t$ which implies that $pz^t=wq$. Hence ${\bf z}\in D_{pz^t}=D_{wq}$.
Put $s_1=r_2=\eps, r_1=z^t, s_2=w$.
\end{proof}


\begin{lemma}\label{l3}
For any $a,b\in\mathcal{M}_{\lambda}\setminus\{0\}$ the set $\{x: ax=b$, or $xa=b\}$ is finite.
\end{lemma}

\begin{proof}
Recall that $\mathcal{S}_{\lambda}$ is the Rees quotient semigroup $(\mathcal{P}_{\lambda}{\times}\mathcal{P}_{\lambda})/J$, where $J=\{0\}{\times}\mathcal{P}_{\lambda}\cup \mathcal{P}_{\lambda}{\times}\{0\}$. The point $q(J)\in \mathcal{S}_{\lambda}$, where $q$ is the quotient map is naturally denoted by $0$.
Fix any non-zero elements $a,b\in \mathcal{S}_{\lambda}$. The definition of $\mathcal{S}_{\lambda}$ ensures that $a=((c,d),(e,f))$, $b=((p,q),(r,s))$ and $\{(a,b),(c,d),(p,q),(r,s)\}\subset \mathcal{P}_{\lambda}\setminus\{0\}$. By Lemma~1 from~\cite{Mesyan-Mitchell-Morayne-Peresse-2013}, the sets $$A=\{x\in \mathcal{P}_{\lambda}: (c,d)\cdot x=(p,q),\hbox{ or } x\cdot(c,d)=(p,q)\}\hbox{ and }$$
$$B=\{x\in \mathcal{P}_{\lambda}: (e,f)\cdot x=(r,s),\hbox{ or } x\cdot(e,f)=(r,s)\}$$ are finite. Then the set
$$\{x\in \mathcal{S}_{\lambda}: ax=b,\hbox{ or }xa=b\}\subset A{\times}B$$ is finite. Recall that the semigroup $\mathcal{M}_{\lambda}$ is a subsemigroup of $\mathcal{S}_{\lambda}$. It follows that for any $a,b\in \mathcal{M}_{\lambda}\setminus\{0\}$ the set
$\{x\in\mathcal{M}_{\lambda}: ax=b$, or $xa=b\}$ is contained in the finite set $\{x\in\mathcal{S}_{\lambda}: ax=b,\hbox{ or }xa=b\}$, witnessing that it is finite as well.
\end{proof}

Groups of automorphism of free inverse semigroups and related structures were investigated by Schein and coauthors in~\cite{MS,MSZ,PS}. Since McAlister semigroups generalize the free inverse semigroup over a singleton, the problem of describing the group $\mathrm{Aut}(\mathcal{M}_{\lambda})$ of automorphisms of the McAlister semigroup $\mathcal{M}_{\lambda}$ arises naturally.

For this we are going to introduce two kinds of automorphisms of $\mathcal{M}_{\lambda}$. Assume that $F_{\lambda}$ is the free monoid over the set $\{a_{\alpha}:\alpha\in\lambda\}$.
Let $f:\lambda\rightarrow \lambda$ be a bijection. Note that $f$ generates the automorphism $h$ of $F_{\lambda}$ defined by $h(a_{\alpha_1}\ldots a_{\alpha_n})=a_{f(\alpha_1)}\ldots a_{f(\alpha_n)}$ and $h(\eps)=\eps$. The automorphism $h$ generates two maps $\phi_f,\psi_f:\mathcal{M}_{\lambda}\rightarrow \mathcal{M}_{\lambda}$ defined by the formulae:
$$\phi_f[((a,b),(c,d))]=((h(a),h(b)),(h(c),h(d))), \quad\hbox{ and }\quad\phi_f(0)=0;$$
$$\psi_f[((a,b),(c,d))]=((h(c),h(d)),(h(a),h(b))), \quad\hbox{ and }\quad\psi_f(0)=0.$$

Let $g$ be an automorphism of $F_{\lambda}$ and fix any $a=a_{\alpha_1}\ldots a_{\alpha_n}\in F_{\lambda}$. Then the following holds:
$$g(a)^t=[g(a_{\alpha_1})\ldots g(a_{\alpha_n})]^t=g(a_{\alpha_n})\ldots g(a_{\alpha_1})=g(a^t).$$
Therefore, if $ac^t=bd^t$, then $$h(a)h(c)^t=h(a)h(c^t)=h(ac^t)=h(bd^t)=h(b)h(d^t)=h(b)h(d)^t.$$ It follows that $$h(c)h(a)^t=[h(a)h(c)^t]^t=[h(b)h(d)^t]^t=h(d)h(b)^t.$$
Hence the maps $\phi_f$ and $\psi_f$ are defined correctly.

Words $a,b\in F_{\lambda}$ are called {\em suffix-incomparable} if neither $a=cb$ nor $b=ca$ for any $c\in F_{\lambda}$.
Note that for any non-zero elements $(a,b)$ and $(c,d)$ of the polycyclic monoid $\mathcal{P}_{\lambda}$, $(a,b)\cdot(c,d)=0$ iff the words $b$ and $c$ are suffix-incomparable.

The next theorem describes the set $\mathrm{Aut}(\mathcal{M}_{\lambda})$ of automorphisms of the semigroup $\mathcal{M}_{\lambda}$.

\begin{proposition}\label{pr}
$\mathrm{Aut}(\mathcal{M}_{\lambda})=\{\phi_f:f$ is a bijection of $\lambda\}\cup \{\psi_f:f$ is a bijection of $\lambda\}$.
\end{proposition}

\begin{proof}
Let us check that for a fixed bijection $f$ the map $\phi_f$ is an automorphism. The injectivity of $\phi_f$ follows from the injectivity of $h$. To prove that the map $\phi_f$ is surjective, fix any non-zero point $((a,b),(c,d))\in\mathcal{M}_{\lambda}$. Clearly, $\phi_f[((h^{-1}(a),h^{-1}(b)),(h^{-1}(c),h^{-1}(d)))]=((a,b),(c,d))$. It remains to check that $$((h^{-1}(a),h^{-1}(b)),(h^{-1}(c),h^{-1}(d)))\in\mathcal{M}_{\lambda}.$$ Since $h^{-1}$ is an automorphism of $F_{\lambda}$ the arguments above imply that $h^{-1}(p)^t=h^{-1}(p^t)$ for any $p\in F_{\lambda}$.
It follows that
$$h^{-1}(a)h^{-1}(c)^t=h^{-1}(a)h^{-1}(c^t)=h^{-1}(ac^t)=h^{-1}(bd^t)=h^{-1}(b)h^{-1}(d^t)=h^{-1}(b)h^{-1}(d)^t,$$
witnessing that $((h^{-1}(a),h^{-1}(b)),(h^{-1}(c),h^{-1}(d)))\in\mathcal{M}_{\lambda}$. Hence the map $\phi_f$ is a bijection.

Let us check that $\phi_f$ is a homomorphism.
Fix any elements $((a,b),(c,d))$ and $((x,y),(u,v))$ of $\mathcal{M}_{\lambda}$ and put ${\bf z}=((a,b),(c,d))\cdot((x,y),(u,v))$. Two cases are possible:
\begin{enumerate}
\item ${\bf z}=0$;
\item ${\bf z}\neq 0$.
\end{enumerate}

In case 1 we have that the words $b$ and $x$ are suffix-incomparable, or the words $d$ and $u$ are suffix-incomparable. It is straightforward to check that suffix-incomparability is preserved by automorphisms of $F_{\lambda}$. Thus, the words $h(b)$ and $h(x)$ are suffix-incomparable, or the words $h(d)$ and $h(u)$ are suffix-incomparable. Hence
$$\phi_f[((a,b),(c,d))]\cdot\phi_f[((x,y),(u,v))]=((h(a),h(b)),(h(c),h(d)))\cdot((h(x),h(y)),(h(u),h(v)))=0=\phi_f({\bf z}).$$

Second case has four subcases:
\begin{itemize}
\item[(2.1)] $x=wb$ and $u=zd$ for some $w,z\in F_{\lambda}$;
\item[(2.2)] $x=wb$ and $d=zu$ for some $w,z\in F_{\lambda}$;
\item[(2.3)] $b=wx$ and $d=zu$ for some $w,z\in F_{\lambda}$;
\item[(2.4)] $b=wx$ and $u=zd$ for some $w,z\in F_{\lambda}$.
\end{itemize}

2.1) Observe that
$${\bf z}=((a,b),(c,d))\cdot((wb,y),(zd,v))=((wa,y),(zc,v)).$$
Then
\begin{equation*}
\begin{split}
\phi_f[((a,b),(c,d))]\cdot\phi_f[((x,y),(u,v))]=((h(a),h(b)),(h(c),h(d)))\cdot((h(wb),h(y)),(h(zd),h(v)))=\\
((h(a),h(b)),(h(c),h(d)))\cdot((h(w)h(b),h(y)),(h(z)h(d),h(v)))=\\
((h(wa),h(y)),(h(zc),h(v)))=\phi_f({\bf z}).
\end{split}
\end{equation*}

2.2) Observe that
$${\bf z}=((a,b),(c,zu))\cdot((wb,y),(u,v))=((wa,y),(c,zv)).$$
Then
\begin{equation*}
\begin{split}
\phi_f[((a,b),(c,d))]\cdot\phi_f[((x,y),(u,v))]=((h(a),h(b)),(h(c),h(zu)))\cdot((h(wb),h(y)),(h(u),h(v)))=\\
=((h(a),h(b)),(h(c),h(z)h(u)))\cdot((h(w)h(b),h(y)),(h(u),h(v)))=\\
((h(wa),h(y)),(h(c),h(zv)))=\phi_f({\bf z}).
\end{split}
\end{equation*}

2.3) Observe that
$${\bf z}=((a,wx),(c,zu))\cdot((x,y),(u,v))=((a,wy),(c,zv)).$$
Then
\begin{equation*}
\begin{split}
\phi_f[((a,b),(c,d))]\cdot\phi_f[((x,y),(u,v))]=((h(a),h(wx)),(h(c),h(zu)))\cdot((h(x),h(y)),(h(u),h(v)))=\\
=((h(a),h(w)h(x)),(h(c),h(z)h(u)))\cdot((h(x),h(y)),(h(u),h(v)))=\\
((h(a),h(wy)),(h(c),h(zv)))=\phi_f({\bf z}).
\end{split}
\end{equation*}

2.4) Observe that
$${\bf z}=((a,wx),(c,d))\cdot((x,y),(zd,v))=((a,wy),(zc,v)).$$
Then
\begin{equation*}
\begin{split}
\phi_f[((a,b),(c,d))]\cdot\phi_f[((x,y),(u,v))]=((h(a),h(wx)),(h(c),h(d)))\cdot((h(x),h(y)),(h(zd),h(v)))=\\
=((h(a),h(w)h(x)),(h(c),h(d)))\cdot((h(x),h(y)),(h(z)h(d),h(v)))=\\
((h(a),h(wy)),(h(zc),h(v)))=\phi_f({\bf z}).
\end{split}
\end{equation*}
Hence for each bijection $f:\lambda\rightarrow \lambda$ the map $\phi_f$ is an automorphism of $\mathcal{M}_{\lambda}$.

Note that for each semigroup $S$ the map $g:S^2\rightarrow S^2$, $g[(x,y)]=(y,x)$ is an automorphism. Moreover, if $J$ is a two sided ideal of $S^2$ such that $(a,b)\in J$ iff $(b,a)\in J$, then $g$ generates an automorphism $g^*$ of the Rees quotient semigroup $S^2/J$ defined by $g^*(a,b)=g(a,b)$ if $(a,b)\notin J$ and $g^*(0)=0$, where by $0$ we denote the image of $J$ under the quotient map. If $X$ is a subsemigroup of $S^2/J$ such that $(a,b)\in X$ iff $(b,a)\in X$, then the restriction $g^*|_X$ is an automorphism of $X$. Applying the above arguments to the polycyclic monoid $\mathcal{P}_{\lambda}$ we get that the map $\psi_{id}$ is an automorphism of the semigroup $\mathcal{M}_{\lambda}$, where $id:\lambda\rightarrow \lambda$ is the identity map. Observe that $\psi_f=\phi_f\circ \psi_{id}=\psi_{id}\circ\phi_f$ for any bijection $f$ on $\lambda$, witnessing that $\psi_f$ is an automorphism of $\mathcal{M}_{\lambda}$. Hence
$$\{\phi_f:f \hbox{ is a bijection of }\lambda\}\cup \{\psi_f:f \hbox{ is a bijection of }\lambda\}\subset \mathrm{Aut}(\mathcal{M}_{\lambda}).$$

To show the converse inclusion fix any automorphism $\phi$ of the McAlister semigroup $\mathcal{M}_{\lambda}$. Let
$$G_1=\{((\eps,a_{\alpha}),(a_{\alpha},\eps)):\alpha\in\lambda\}\quad\hbox{ and }\quad G_2=\{((a_{\alpha},\eps),(\eps,a_{\alpha})):\alpha\in\lambda\}.$$
\begin{claim}\label{claim}
The set $G_1\cup G_2\cup\{0\}$ generates $\mathcal{M}_{\lambda}$
\end{claim}
\begin{proof}
By $\langle G_1\cup G_2\cup\{0\}\rangle$ we denote the subsemigroup of $\mathcal{M}_{\lambda}$ which is generated by $G_1\cup G_2\cup\{0\}$. Observe that for each word $u=u_1\ldots u_n\in F_{\lambda}\setminus\{\eps\}$,
$$((\eps,u),(u^t,\eps))=((\eps,u_1),(u_1,\eps))\cdot((\eps,u_2),(u_2,\eps))\cdots ((\eps,u_n),(u_n,\eps))\quad \hbox{and}$$
$$((u,\eps),(\eps,u^t))=((u_n,\eps),(\eps,u_n))\cdot((u_{n-1},\eps),(\eps,u_{n-1}))\cdots ((u_1,\eps),(\eps,u_1)).$$
Hence $\{((\eps,u),(u^t,\eps)),((u,\eps),(\eps,u^t))\}\subset \langle G_1\cup G_2\cup\{0\}\rangle$ for each $u\in F_{\lambda}\setminus\{\eps\}$.
Fix any non-zero element $g\in\mathcal{M}_{\lambda}$. By Lemma~\ref{l0}, we can represent $g$ as $((a,aw),(bw^t,b))$ or $((aw,a),(b,bw^t))$ for some $a,b,w\in F_{\lambda}$.
Assume that $g=((a,aw),(bw^t,b))$. Let $x=((\eps,wb^t),(bw^t,\eps))$, $y=((awb^t,\eps),(\eps,bw^ta^t))$ and $z=((\eps,aw),(w^ta^t,\eps))$. By the preceding arguments, $\{x,y,z\}\subset \langle G_1\cup G_2\cup\{0\}\rangle$.  It remains to calculate that
\begin{equation*}
\begin{split}
xyz=((\eps,wb^t),(bw^t,\eps))\cdot((awb^t,\eps),(\eps,bw^ta^t))\cdot((\eps,aw),(w^ta^t,\eps))=\\
((a,\eps),(bw^t,bw^ta^t))\cdot((\eps,aw),(w^ta^t,\eps))=((a,aw),(bw^t,b))=g.
\end{split}
\end{equation*}
If $g=((aw,a),(b,bw^t))$, then let $x=((aw,\eps),(\eps, w^ta^t))$, $y=((\eps, awb^t),(bw^ta^t,\eps))$ and $z=((wb^t,\eps),(\eps,bw^t))$. Clearly, $\{x,y,z\}\subset \langle G_1\cup G_2\cup\{0\}\rangle$.  It remains to calculate that
\begin{equation*}
\begin{split}
xyz=((aw,\eps),(\eps, w^ta^t))\cdot((\eps, awb^t),(bw^ta^t,\eps))\cdot((wb^t,\eps),(\eps,bw^t))=\\
((aw,awb^t),(b,\eps))\cdot((wb^t,\eps),(\eps,bw^t))=((aw,a),(b,bw^t))=g.
\end{split}
\end{equation*}
Thus $g\in \langle G_1\cup G_2\cup\{0\}\rangle$. Hence $\mathcal{M}_{\lambda}=\langle G_1\cup G_2\cup\{0\}\rangle$.
\end{proof}
Let us note that in the previous claim we add $0$ to the set of generators of $\mathcal{M}_{\lambda}$ only to establish the case $\lambda=1$. If $\lambda>1$ this is not necessary.
\begin{claim}
 $\phi(G_1\cup G_2)=G_1\cup G_2$.
\end{claim}
\begin{proof}
Let $M$ be the set of all maximal elements of the semilattice $E(\mathcal{M}_{\lambda})$. It is easy to see that $e\in M$ iff $e=((\eps,\eps),(a_{\alpha},a_{\alpha}))$ or $e=((a_{\alpha},a_{\alpha}),(\eps,\eps))$ for some $\alpha\in \lambda$. Since the restriction of $\phi$ on the semilattice $E(\mathcal{M}_{\lambda})$ is an automorphism of $E(\mathcal{M}_{\lambda})$ we deduce that $\phi|_M$ is a bijection of $M$.

First we shall show that $\phi (G_1)\subset G_1\cup G_2$. Let us note that
$$G_1\cup G_2=\{((a,b),(c,d))\in\mathcal{M}_{\lambda}:|ac^t|=1\}\setminus M.$$ Fix any $g=((\eps,a_{\alpha}),(a_{\alpha},\eps))\in G_1$ and let $((a,b),(c,d))=\phi(g)$.  To derive a contradiction, assume that $|ac^t|=|bd^t|>1$. Then,
$$\phi[((\eps,\eps),(a_{\alpha},a_{\alpha}))]=\phi(gg^{-1})=\phi(g)\cdot \phi(g)^{-1}=((a,b),(c,d))\cdot((b,a),(d,c))=((a,a),(c,c))\notin M.$$
But this contradicts to the fact $\phi(M)=M$. The obtained contradiction implies that $|ac^t|=|bd^t|=1$. Since $\phi$ is an automorphism and $g\notin E(\mathcal{M}_{\lambda})$ we deduce that $\phi(g)\notin E(\mathcal{M}_{\lambda})$. Consequently, $\phi(g)\in G_1\cup G_2$, witnessing that $\phi(G_1)\subset G_1\cup G_2$. Similarly, it can be showed that $\phi(G_2)\subset G_1\cup G_2$. It follows that $\phi(G_1\cup G_2)\subset G_1\cup G_2$.
Since $\phi^{-1}$ is an automorphism of $\mathcal{M}_{\lambda}$, the same arguments imply that $\phi^{-1}(G_1\cup G_2)\subset G_1\cup G_2$. Then $G_1\cup G_2=\phi(\phi^{-1}(G_1\cup G_2))\subset \phi(G_1\cup G_2)$. Hence $\phi(G_1\cup G_2)=G_1\cup G_2$.
\end{proof}
\begin{claim}
$\phi(G_1)=G_1$ and $\phi(G_2)=G_2$, or $\phi(G_1)=G_2$ and $\phi(G_2)=G_1$.
\end{claim}
\begin{proof}
Two cases are possible:
\begin{enumerate}
\item $\phi(G_1)\subset G_1$;
\item $\phi(G_1)\cap G_2\neq \emptyset$.
\end{enumerate}

1. If $\phi(G_1)=G_1$, then $\phi(G_2)=\phi(G_1^{-1})=\phi(G_1)^{-1}=G_1^{-1}=G_2$ and we are done. Assuming that $G_1\setminus \phi(G_1)\neq \emptyset$ and taking into account that $\phi(G_1\cup G_2)=G_1\cup G_2$, we obtain that there exists $s=((a_{\alpha},\eps),(\eps,a_{\alpha}))\in G_2$ such that $\phi(s)=((\eps,a_{\beta}),(a_{\beta},\eps))\in G_1$. Since $\phi(G_1\cup G_2)=G_1\cup G_2$, there exists $t=((a_{\gamma},\eps),(\eps,a_{\gamma}))\in G_2$ such that $\phi(t)=s$. Note that $\alpha\neq \beta$, or equivalently $s\neq \phi(s)^{-1}$, because $s\in G_2$ and
$$\phi(s)^{-1}=\phi(s^{-1})=\phi[((\eps,a_{\alpha}),(a_{\alpha},\eps))]\in \phi(G_1)\subset G_1.$$
Also note that $st=((a_{\gamma}a_{\alpha},\eps),(\eps,a_{\alpha}a_{\gamma}))\neq 0$ which implies that $\phi(st)\neq 0$ as well. On the other hand
$$\phi(st)=\phi(s)\cdot \phi(t)= ((\eps,a_{\beta}),(a_{\beta},\eps))\cdot((a_{\alpha},\eps),(\eps,a_{\alpha}))=0.$$
The obtained contradiction implies that if case 1 holds, then $\phi(G_1)=G_1$ and $\phi(G_2)=G_2$.

2. If $G_2\subset \phi(G_1)$, then $\phi(G_2)=\phi(G_1^{-1})=\phi(G_1)^{-1}\supset G_2^{-1}=G_1$. Since $\phi$ is a bijection and $\phi(G_1\cup G_2)=G_1\cup G_2$ we get that $\phi(G_1)=G_2$ and $\phi(G_2)=G_1$. Assume that $G_2\setminus \phi(G_1)\neq \emptyset$. Since $\phi(G_1)\cap G_2\neq \emptyset$,
there exists $s=((\eps,a_{\alpha}),(a_{\alpha},\eps))\in G_1$ such that $\phi(s)=((a_{\beta},\eps),(\eps,a_{\beta}))\in G_2$. Since $\phi(G_1\cup G_2)=G_1\cup G_2$ there exists $t=((a_{\gamma},\eps),(\eps,a_{\gamma}))\in G_2$ such that $\phi(t)=((a_{\delta},\eps),(\eps,a_{\delta}))\in G_2\setminus \phi(G_1)$. Note that $\gamma\neq \alpha$, or equivalently $t\neq s^{-1}$, because $\phi(t)\in G_2$ and
$$\phi(s^{-1})=\phi(s)^{-1}=((\eps,a_{\beta}),(a_{\beta},\eps))\in G_1.$$
Then
$$0=\phi(0)=\phi(st)=\phi(s)\cdot\phi(t)=((a_{\beta},\eps),(\eps,a_{\beta}))\cdot((a_{\delta},\eps),(\eps,a_{\delta}))=((a_{\delta}a_{\beta},\eps),(\eps,a_{\beta}a_{\delta}))\neq 0.$$
The obtained contradiction implies that if case 2 holds, then $\phi(G_1)=G_2$ and $\phi(G_2)=G_1$.
\end{proof}

So, we either have $\phi(G_1)=G_1$ and $\phi(G_2)=G_2$, or $\phi(G_1)=G_2$ and $\phi(G_2)=G_1$.
In the first case $\phi$ generates a bijection $f:\lambda\rightarrow \lambda$ defined by:
$f(\alpha)=\beta$ iff $\phi[((\eps,a_{\alpha}),(a_{\alpha},\eps))]=((\eps,a_{\beta}),(a_{\beta},\eps))$. At this point it is straightforward to check that the restriction of $\phi$ on $G_1\cup G_2\cup\{0\}$ coincides with the restriction of $\phi_f$ on $G_1\cup G_2\cup\{0\}$. Recall that by Claim~\ref{claim} the set $G_1\cup G_2\cup\{0\}$ generates $\mathcal{M}_{\lambda}$. Since both $\phi$ and $\phi_f$ are homomorphisms we get that $\phi=\phi_f$.

In the second case  $\phi$ generates a bijection $f:\lambda\rightarrow \lambda$ defined by:
$f(\alpha)=\beta$ iff $\phi[((\eps,a_{\alpha}),(a_{\alpha},\eps))]=((a_{\beta},\eps),(\eps,a_{\beta}))$. At this point it is straightforward to check that the restriction of $\phi$ on $G_1\cup G_2\cup\{0\}$ coincides with the restriction of $\psi_f$ on $G_1\cup G_2\cup\{0\}$, witnessing that $\phi=\psi_f$.

Hence $\mathrm{Aut}(\mathcal{M}_{\lambda})=\{\phi_f:f$ is a bijection of $\lambda\}\cup \{\psi_f:f$ is a bijection of $\lambda\}$.
\end{proof}

By $Sym(\lambda)$ we denote the group of permutations of the cardinal $\lambda$, and by $\mathbb Z_2$ the two-element group $\{1,-1\}$.
Due to  Mashevitzky, Schein and Zhitomirski~\cite{MSZ}, the group of automorphisms of the free inverse semigroup $\mathbf{F}_X$ over a set $X$ of cardinality $\lambda$ is isomorphic to the wreath product of $Sym(\lambda)$ and $\mathbb{Z}_2$. In particular, if the set $X$ is singleton, then the group $\mathrm{Aut}(\mathbf{F}_X)$ is isomorphic to $\mathbb{Z}_2$. The next theorem indicates that McAlister semigroups are related not only to the free inverse semigroup over a singleton, but also to arbitrary free inverse semigroups.

\begin{theorem}\label{t}
The group $\mathrm{Aut}(\mathcal{M}_{\lambda})$ is isomorphic to the direct product $Sym(\lambda){\times}\mathbb{Z}_2$.
\end{theorem}

\begin{proof}
The isomorphism $i:\mathrm{Aut}(\mathcal{M}_{\lambda})\rightarrow Sym(\lambda){\times}\mathbb{Z}_2$ is defined by the following formulae:
$$i(\phi_f)=(f,1)\quad\hbox{ and }\quad i(\psi_f)=(f,-1), \quad \hbox{ for each bijection }f\in Sym(\lambda).$$
Let us sketch why the map $i$ is an isomorphism. By Proposition~\ref{pr}, $i$ is a bijection. Recall that the set $G_1\cup G_2\cup \{0\}$ generates $\mathcal{M}_{\lambda}$ (see Claim~\ref{claim} in the proof of Proposition~\ref{pr}). Observe that
$$(\phi_f\circ \phi_g)|_{G_1\cup G_2\cup \{0\}}=\phi_{f\circ g}|_{G_1\cup G_2\cup \{0\}}.$$
Since $\phi_f\circ \phi_g$ as well as $\phi_{f\circ g}$ are automorphisms of $\mathcal{M}_{\lambda}$ we deduce that $\phi_f\circ \phi_g=\phi_{f\circ g}$. It follows that
$$i(\phi_f)i(\phi_g)=(f,1)(g,1)=(f\circ g,1)=i(\phi_{f\circ g})=i(\phi_f\circ \phi_g).$$
Observe that $(\psi_f\circ \psi_g)|_{G_1\cup G_2\cup \{0\}}=\phi_{f\circ g}|_{G_1\cup G_2\cup \{0\}}$. It follows that $\psi_f\circ \psi_g=\phi_{f\circ g}$. Then
$$i(\psi_f)i(\psi_g)=(f,-1)(g,-1)=(f\circ g,1)=i(\phi_{f\circ g})=i(\psi_f\circ \psi_g).$$
Repeating the same arguments, one can check that $\psi_f\circ \phi_g=\psi_{f\circ g}$ and $\phi_f\circ \psi_g=\psi_{f\circ g}$. Then
$$i(\psi_f)i(\phi_g)=(f,-1)(g,1)=(f\circ g,-1)=i(\psi_{f\circ g})=i(\psi_f\circ \phi_g) \qquad \hbox{and}$$
$$i(\phi_f)i(\psi_g)=(f,1)(g,-1)=(f\circ g,-1)=i(\psi_{f\circ g})=i(\phi_f\circ \psi_g).$$
Hence the map $i$ is an isomorphism.
\end{proof}

Since the semigroup $\mathcal{M}_1$ is isomorphic to the free inverse semigroup over a singleton with adjoined zero, Theorem~\ref{t} implies the following known result.

\begin{corollary}
The group of automorphisms of the free inverse semigroup over a singleton is isomorphic to $\mathbb{Z}_2$.
\end{corollary}

\section{Topological properties of McAlister semigroups}

A word $a$ of $F_{\lambda}$ is called a {\em suffix} of a word $b\in F_{\lambda}$ if there exists $c\in F_{\lambda}$ such that $b=ca$.
\begin{theorem}\label{iso}
Each non-zero element of a semitopological semigroup $\mathcal{M}_{\lambda}$ is isolated.
\end{theorem}

\begin{proof}
First, assume that $\lambda=1$ and the free monoid $F_1$ is taken over the singleton $\{a\}$. Fix any non-zero element $((a^n,a^m),(a^k,a^l))\in \mathcal M_1$. Put $t=\max\{n,m,k,l\}+1$. Note that the elements $x=((a^t,a^t),(\eps,\eps))$ and $y=((\eps,\eps),(a^t,a^t))$ are idempotents in $\mathcal M_1$. Since shifts are continuous on the Hausdorff space $\mathcal M_1$, the retracts
$$R_1=x\cdot \mathcal M_1, \qquad R_2=\mathcal{M}_1\cdot x, \qquad R_3=y\cdot\mathcal{M}_1\quad \hbox{ and }\quad R_4=\mathcal{M}_1\cdot y$$
are closed subsets of $\mathcal{M}_1$. It is easy to check that the mentioned above retracts do not contain the point $((a^n,a^m),(a^k,a^l))$. It follows that the set $U=\mathcal{M}_1\setminus (\cup_{i=1}^4R_i)$ is an open neighborhood of $((a^n,a^m),(a^k,a^l))$. Let us show that the set $U$ is finite. Assume that $z=((a^p,a^q),(a^r,a^s))\in U$. Since $z\notin R_1$ we deduce that $p<t$, because otherwise $z=xz\in x\cdot \mathcal{M}_1=R_1$. Since $z\notin R_2$ we get that $q<t$, because otherwise $z=zx\in \mathcal{M}_1\cdot x=R_2$. Similarly one can check that the assumption $z\notin R_3$ provides that $r<t$. Finally, we deduce that $s<t$ from the condition $z\notin R_4$. It follows that $U\subset \{((a^p,a^q),(a^r,a^s))\in\mathcal{M}_1: \max\{p,q,r,s\}<t\}$, witnessing that the set $U$ is finite. The Hausdorffness of $\mathcal{M}_1$ implies that the point $((a^n,a^m),(a^k,a^l))$ is isolated.

Assuming that $\lambda\geq 2$, fix any non-zero point $((x,y),(u,v))$ of the McAlister semigroup $\mathcal{M}_{\lambda}$ and two distinct words $z_1$ and $z_2$ in $F_{\lambda}$ such that $|z_1|=|z_2|=1$. Observe that the following statements hold:
\begin{enumerate}
\item $((x,y),(u,v))\cdot((z_1y,z_1y),(z_1v,z_1v))=((z_1x,z_1y),(z_1u,z_1v))\neq 0$;
\item $((x,y),(u,v))\cdot((z_2y,z_2y),(z_2v,z_2v))=((z_2x,z_2y),(z_2u,z_2v))\neq 0$;
\item $((z_1x,z_1x),(z_1u,z_1u))\cdot ((x,y),(u,v))=((z_1x,z_1y),(z_1u,z_1v))\neq 0$;
\item $((z_2x,z_2x),(z_2u,z_2u))\cdot ((x,y),(u,v))=((z_2x,z_2y),(z_2u,z_2v))\neq 0.$
\end{enumerate}
Since $\mathcal{M}_{\lambda}$ is a semitopological semigroup,
there exist open neighborhoods $U_1$, $U_2$, $U_3$, $U_4$ of $((x,y),(u,v))$ such that
 $$U_1\cdot((z_1y,z_1y),(z_1v,z_1v))\cup U_2\cdot((z_2y,z_2y),(z_2v,z_2v))\subset \mathcal{M}_{\lambda}\setminus\{0\}\hbox{ and}$$
$$((z_1x,z_1x),(z_1u,z_1u))\cdot U_3\cup ((z_2x,z_2x),(z_2u,z_2u))\cdot U_4\subset \mathcal{M}_{\lambda}\setminus\{0\}.$$
Let $U=\cap_{i=1}^4U_i$ and fix any non-zero element $((a,b),(c,d))\in U$. By the choice of $U$, the following statements hold:
\begin{enumerate}
\item $((a,b),(c,d))\cdot((z_1y,z_1y),(z_1v,z_1v))\neq 0$;
\item $((a,b),(c,d))\cdot((z_2y,z_2y),(z_2v,z_2v))\neq 0$;
\item $((z_1x,z_1x),(z_1u,z_1u))\cdot ((a,b),(c,d))\neq 0$;
\item $((z_2x,z_2x),(z_2u,z_2u))\cdot ((a,b),(c,d))\neq 0$.
\end{enumerate}
The above statements imply that the following conditions hold:
\begin{enumerate}
\item $[b$ is a suffix of $z_1y$, or $z_1y$ is a suffix of $b]$ and $[d$ is a suffix of $z_1v$, or $z_1v$ is a suffix of $d]$;
\item $[b$ is a suffix of $z_2y$, or $z_2y$ is a suffix of $b]$ and $[d$ is a suffix of $z_2v$, or $z_2v$ is a suffix of $d]$;
\item $[a$ is a suffix of $z_1x$, or $z_1x$ is a suffix of $a]$ and $[c$ is a suffix of $z_1u$, or $z_1u$ is a suffix of $c]$;
\item $[a$ is a suffix of $z_2x$, or $z_2x$ is a suffix of $a]$ and $[c$ is a suffix of $z_2u$, or $z_2u$ is a suffix of $c]$.
\end{enumerate}

Note that words $z_1y$ and $z_2y$ are different and $|z_1y|=|z_2y|$. Then the following cases are not possible:
\begin{itemize}
\item $z_1y$ is a suffix of $b$ and $z_2y$ is a suffix of $b$;
\item $z_1y$ is a suffix of $b$ and $b$ is a suffix of $z_2y$;
\item $b$ is a suffix of $z_1y$ and $z_2y$ is a suffix of $b$.
\end{itemize}
Hence conditions 1 and 2 imply that $b$ is a suffix of $z_1y$ and $z_2y$, witnessing that $b$ is a suffix of $y$.

Since the words $z_1v$ and $z_2v$ are different and $|z_1v|=|z_2v|$, the following cases are impossible:

\begin{itemize}
\item $z_1v$ is a suffix of $d$ and $z_2v$ is a suffix of $d$;
\item $z_1v$ is a suffix of $d$ and $d$ is a suffix of $z_2v$;
\item $d$ is a suffix of $z_1v$ and $z_2v$ is a suffix of $d$.
\end{itemize}
Hence conditions 1 and 2 provide that $d$ is a suffix of $v$.
Using similar arguments, one can check that conditions 3 and 4 ensure that $a$ is a suffix of $x$ and $c$ is a suffix of $u$.
It remains to observe that for a fixed words $x,y,u,v\in F_{\lambda}$ there are only finitely many elements $((a,b),(c,d))\in\mathcal{M}_{\lambda}$ such that $a$ is a suffix of $x$, $b$ is a suffix of $y$, $c$ is a suffix $u$ and $d$ is a suffix of $v$. Hence the neighborhood $U$ is finite, witnessing that the point $((x,y),(u,v))$ is isolated.
\end{proof}

\begin{corollary}\label{c2}
A semitopological free inverse semigroup over a singleton is discrete.
\end{corollary}

A word $x\in F_{\lambda}$ is called a {\em subword} of a word $y\in F_{\lambda}$ if there exist $a,b\in F_{\lambda}$ such that $y=axb$.

Despite the fact that compact topological semigroups cannot contain neither bicyclic monoid nor polycyclic monoids~\cite{Anderson-Hunter-Koch-1965,BardGut-2016(1)}, the next theorem shows that McAlister semigroups admit compact inverse semigroup topology.

\begin{theorem}\label{c}
There exists the unique topology $\tau$ on $\mathcal{M}_{\lambda}$ such that $(\mathcal{M}_{\lambda},\tau)$ is a compact topological inverse semigroup.
\end{theorem}

\begin{proof}
The topology $\tau$ is defined as follows:
\begin{itemize}
\item each non-zero element of $\mathcal{M}_{\lambda}$ is isolated;
\item the family $\{U_{A}:A$ is a finite subset of $F_{\lambda}\setminus\{\eps\}\}$ forms an open neighborhood base at $0$, where $U_A=\mathcal{M}_{\lambda}\setminus (\cup_{a\in A}D_a)$.
\end{itemize}

Obviously, the space $(\mathcal{M}_{\lambda},\tau)$ is Hausdorff. Since each $\mathscr D$-class in $\mathcal{M}_{\lambda}$ is finite (recall Lemma~\ref{l2}), the space $(\mathcal{M}_{\lambda},\tau)$ is compact. Since $D_a^{-1}=D_a$ for each $a\in F_{\lambda}\setminus\{\eps\}$, we have that $U_A^{-1}=U_A$ for each finite subset $A\subset F_{\lambda}\setminus\{\eps\}$. It follows that the inversion is continuous in $(\mathcal{M}_{\lambda},\tau)$. To prove that $(\mathcal{M}_{\lambda},\tau)$ is a topological semigroup it is sufficient to show that for each basic open neighborhood $U_A$ of $0$ and $x\in \mathcal{M}_{\lambda}\setminus\{0\}$ there exists a basic open neighborhood $U_B$ of $0$ such that $x\cdot U_B\cup U_B\cdot x\cup U_B\cdot U_B\subset U_A$. So, fix any basic open neighborhood $U_A$ of $0$, where $A=\{a_1,\ldots,a_n\}\subset F_{\lambda}\setminus\{\eps\}$. Let
$$B=\{b:\hbox{ there exists }i\leq n\hbox{ such that }b \hbox{ is a subword of }a_i\}\setminus\{\eps\}.$$
We claim that $U_B$ is a two-sided ideal of $\mathcal{M}_{\lambda}$. Indeed, fix any $x\in\mathcal{M}_{\lambda}$ and $y\in D_q\subset U_B$. Note that $q\notin B$.
If $xy\neq 0$, then Proposition~\ref{p1} implies that $xy\in D_{sqr}$ where $s,r\in F_{\lambda}$. By the choice of $B$, $sqr\in F_{\lambda}\setminus B$. Thus $xy\in D_{sqr}\subset U_B$.
Similarly, one can check that $yx\in U_B$. Hence $U_B$ is a two-sided ideal of $\mathcal{M}_{\lambda}$. At this point it is easy to see that $x\cdot U_B\cup U_B\cdot x\cup U_B\cdot U_B\subset U_B\subset U_A$, witnessing that $(\mathcal{M}_{\lambda},\tau)$ is a topological inverse semigroup.
The uniqueness of the topology $\tau$ follows from Theorem~\ref{iso}.
\end{proof}

Eberhart and Selden~\cite{Eberhart-Selden-1969} showed that if the bicyclic monoid ${\mathcal{B}}$ is a dense proper subsemigroup of a topological semigroup $S$, then $I=S\setminus{\mathcal{B}}$ is a two-sided ideal of the semigroup $S$. Mesyan, Mitchell, Morayne and P\'{e}resse~\cite{Mesyan-Mitchell-Morayne-Peresse-2013} showed that if a polycyclic monoid $\mathcal{P}_{\lambda}$ is a dense proper subsemigroup of a topological semigroup $S$, then $I=(S\setminus{\mathcal{P}_{\lambda}})\cup\{0\}$ is a two-sided ideal of the semigroup $S$. Note that the above statements still hold if $S$ is a semitopological semigroup~\cite{BardGut-2016(1)}. The following lemma shows that McAlsiter semigroups also possess this property.

\begin{lemma}
Let $S$ be a semitopological semigroup which contains $\mathcal{M}_{\lambda}$ as a dense proper subsemigroup. Then $(S\setminus \mathcal{M}_{\lambda})\cup\{0\}$ is a two-sided ideal in $S$.
\end{lemma}

\begin{proof}
Since $\mathcal{M}_{\lambda}$ is dense in $S$ and $S$ is a semitopological semigroup, we obtain that $0s=s0=0$ for any $s\in S$.
Let us show that $\mathcal{M}_{\lambda}\cdot (S\setminus\mathcal{M}_{\lambda})\subset (S\setminus\mathcal{M}_{\lambda})\cup\{0\}$.
Assuming the contrary, fix $s\in S\setminus\mathcal{M}_{\lambda}$ and $x\in \mathcal{M}_{\lambda}\setminus\{0\}$ such that
$y=xs\in \mathcal{M}_{\lambda}\setminus\{0\}$. By Theorem~\ref{iso}, $y$ is an isolated point in $\mathcal{M}_{\lambda}$. Since $\mathcal{M}_{\lambda}$ is dense in $S$ we get that $y$ is isolated in $S$. The continuity of left shifts in $S$ implies that there exists an open neighborhood $U$ of $s$ such that $x\cdot U=\{y\}$. Since $s$ is an accumulation point of the set $\mathcal{M}_{\lambda}$, we have that the set $U\cap\mathcal{M}_{\lambda}$ is infinite. But this contradicts Lemma~\ref{l3}. The obtained contradiction implies that $\mathcal{M}_{\lambda}\cdot (S\setminus\mathcal{M}_{\lambda})\subset (S\setminus\mathcal{M}_{\lambda})\cup\{0\}$.

Now let us show that $(S\setminus \mathcal{M}_{\lambda})\cup\{0\}$ is a left ideal in $S$. At this point it is sufficient to prove that $st\in (S\setminus \mathcal{M}_{\lambda})\cup\{0\}$ for any $s,t\in S\setminus \mathcal{M}_{\lambda}$. To derive a contradiction, assume that there exist elements $s,t\in S\setminus \mathcal{M}_{\lambda}$ such that $y=st\in \mathcal{M_{\lambda}}\setminus\{0\}$. Recall that the point $y$ is isolated in $S$. Since $S$ is a semitopological semigroup, there exists an open neighborhood $U$ of $s$ such that $U\cdot t=\{y\}$. The density of $\mathcal{M}_{\lambda}$ in $S$ yields that the open set $U$ contains a point $x\in \mathcal{M}_{\lambda}\setminus\{0\}$. Then $xt=y\in\mathcal{M}_{\lambda}\setminus\{0\}$ which contradicts to the inclusion $\mathcal{M}_{\lambda}\cdot (S\setminus\mathcal{M}_{\lambda})\subset (S\setminus\mathcal{M}_{\lambda})\cup\{0\}$. Hence $(S\setminus \mathcal{M}_{\lambda})\cup\{0\}$ is a left ideal in $S$.
Similarly it can be showed that $(S\setminus \mathcal{M}_{\lambda})\cup\{0\}$ is a right ideal in $S$.
\end{proof}

\section{Locally compact semitopological McAlister semigroups}
Locally compact topological graph inverse semigroups were investigated by Mesyan, Mitchell, Moray\-ne and P\'{e}resse in~\cite{Mesyan-Mitchell-Morayne-Peresse-2013}. They showed that a locally compact topological graph inverse semigroup over a finite graph $E$ is discrete. This result implies that a locally compact topological polycyclic monoids $\mathcal{P}_k$, $k\in\mathbb{N}$ are discrete. Gutik and the author~\cite{BardGut-2016(1)} showed that for every non-zero cardinal $\lambda$ a locally compact topological polycyclic monoid $\mathcal{P}_{\lambda}$ is discrete. Graph inverse semigroups which admit only the discrete locally compact semigroup topology were characterized in~\cite{Bardyla-2017(1)}. In~\cite{Gutik-2015} Gutik showed that a locally compact semitopological bicyclic monoid with adjoined zero (i.e., polycyclic monoid $\mathcal{P}_1$) is either compact or discrete. It turns out that this kind of dichotomy also holds for some generalizations of the bicyclic monoid, including all polycyclic monoids.

\begin{proposition}\label{pro}
Assume that $\mathcal{M}_{k}$ is a locally compact semitopological semigroup for some positive integer $k$. Let $U$ be a compact open neighborhood of $0$ and $A_U=\{a\in F_{\lambda}: D_a\cap U\neq \emptyset\}$. Then $D_a\subset U$ for all but finitely many $a\in A_U$.
\end{proposition}

\begin{proof}
Assume that $\mathcal{M}_{k}$ is a locally compact semitopological semigroup. Theorem~\ref{iso} implies that all non-zero points in $\mathcal{M}_{k}$ are isolated. Consequently, each open neighborhood of $0$ is closed. Let $U$ be any compact neighborhood of $0$. The compactness of $U$ implies that for each open neighborhood $V\subset U$ of $0$ the set $U\setminus V$ is finite.

Enumerate $A_U=\{a_n:n\in\mathbb{N}\}$. To derive a contradiction, assume that there exists an infinite subset $B\subset \mathbb{N}$ such that $D_{a_n}\setminus U\neq \emptyset$ for each $n\in B$. Let us imagine $\mathscr{D}$-classes $D_{a_n}, n\in B$ using the classical ``egg-box'' picture~\cite[Chapter 2]{Clifford-Preston-1961-1967}. Then $D_{a_n}$ can be associated with the set $S_n=\{(i,j):i,j\leq |a_n|+1\}$ (recall Lemma~\ref{l2}). If for some set $A$ (in our case the role of the set $A$ will play $U$) the sets $S_n\setminus A$ and $S_n\cap A$ are nonempty, then there exists an element $(i,j)\in A\cap S_n$ such that at least one of
its ``neighbors'' is not in $A$. That is
$$(i-1,j)\in S_n\setminus A,\quad\hbox{ or }\quad (i,j-1)\in S_n\setminus A,\quad\hbox{ or }\quad(i+1,j)\in S_n\setminus A,\quad\hbox{ or }\quad(i,j+1)\in S_n\setminus A.$$
Note that the notion of ``neighbor'' in $D_{a_n}$ is a bit different. For an element $$((b_1\ldots b_n,c_1\ldots c_m),(d_1\ldots d_k,e_1\ldots e_p))\in D_{a_n}$$
the set of its ``neighbors'' consists of the following four elements:
\begin{itemize}
\item $((b_1\ldots b_{n-1},c_1\ldots c_m),(d_1\ldots d_kb_n,e_1\ldots e_p))$;
\item $((b_1\ldots b_{n},c_1\ldots c_{m-1}),(d_1\ldots d_k,e_1\ldots e_pc_m))$;
\item $((b_1\ldots b_{n}d_k,c_1\ldots c_m),(d_1\ldots d_{k-1},e_1\ldots e_p))$;
\item $((b_1\ldots b_{n},c_1\ldots c_me_p),(d_1\ldots d_k,e_1\ldots e_{p-1}))$.
\end{itemize}
Hence if we formalize the preceeding arguments, then we obtain that for each $n\in B$ there exist words $b_n,c_n,d_n,e_n\in F_{k}$ and a word $x_n\in F_k$ with $|x_n|=1$ such that at least one of the following cases hold:
\begin{enumerate}
\item $((b_nx_n,c_n),(d_n,e_n))\in U\cap D_{a_n}\quad$ and $\quad((b_n,c_n),(d_nx_n,e_n))\in D_{a_n}\setminus U;$

\item $((b_n,c_nx_n),(d_n,e_n))\in U\cap D_{a_n}\quad$ and $\quad ((b_n,c_n),(d_n,e_nx_n))\in D_{a_n}\setminus U;$

\item $((b_n,c_n),(d_nx_n,e_n))\in U\cap D_{a_n}\quad$ and $\quad ((b_nx_n,c_n),(d_n,e_n))\in D_{a_n}\setminus U;$

\item $((b_n,c_n),(d_n,e_nx_n))\in U\cap D_{a_n}\quad$ and $\quad ((b_n,c_nx_n),(d_n,e_n))\in D_{a_n}\setminus U.$
\end{enumerate}
The Pigeonhole principle implies that there exist an infinite subset $C\subset B$ and $i\in \{1,2,3,4\}$ such that for each $n\in C$ case $i$ holds. Since the free monoid $F_{k}$ contains precisely $k$ words of length $1$, using one more time the Pigeonhole principle we can find an infinite subset $D\subset C$ such that the words $x_n$ and $x_m$ coincide, whenever $n,m\in D$. For each $n\in D$ let us denote $x_n$ simply by $x$.
At this point we have four possibilities:

$(i=1)$ In this case for each $n\in D$
$$((b_nx,c_n),(d_n,e_n))\in U\cap D_{a_n}\quad \hbox{ and }\quad ((b_n,c_n),(d_nx,e_n))\notin U\cap D_{a_n}.$$
Consider the product $((\eps,x),(x,\eps))\cdot 0=0$. Since the semigroup $\mathcal{M}_k$ is semitopological, there exists an open neighborhood $V\subset U$ of $0$ such that $((\eps,x),(x,\eps))\cdot V\subset U$. Observe that for each $n\in D$
$$((\eps,x),(x,\eps))\cdot ((b_nx,c_n),(d_n,e_n))=((b_n,c_n),(d_nx,e_n))\notin U.$$
It follows that $((b_nx,c_n),(d_n,e_n))\notin V$ for each $n\in D$.
Since the elements $((b_nx,c_n),(d_n,e_n))$, $n\in D$ belong to different $\mathscr D$-classes, we get that
the set $T=\{((b_nx,c_n),(d_n,e_n)):n\in D\}$ is infinite. Moreover, $T$ is contained in $U\setminus V$ which contradicts to the fact that the set $U\setminus V$ is finite.

$(i=2)$ In this case for each $n\in D$
$$((b_n,c_nx),(d_n,e_n))\in U\cap D_{a_n}\quad \hbox{ and }\quad ((b_n,c_n),(d_n,e_nx))\notin U\cap D_{a_n}.$$
Consider the product $0\cdot ((x,\eps),(\eps,x))=0$. There exists an open neighborhood $V\subset U$ of $0$ such that $V\cdot((x,\eps),(\eps,x))\subset U$. Observe that for each $n\in D$
$$((b_n,c_nx),(d_n,e_n))\cdot((x,\eps),(\eps,x))=((b_n,c_n),(d_n,e_nx))\notin U.$$
Hence the infinite set $\{((b_n,c_nx),(d_n,e_n)):n\in D\}$ is contained in $U\setminus V$ which implies a contradiction.

$(i=3)$ In this case for each $n\in D$
$$((b_n,c_n),(d_nx,e_n))\in U\cap D_{a_n}\quad \hbox{ and }\quad ((b_nx,c_n),(d_n,e_n))\notin U\cap D_{a_n}.$$
Consider the product $((x,\eps),(\eps,x))\cdot 0=0$. There exists an open neighborhood $V\subset U$ of $0$ such that $((x,\eps),(\eps,x))\cdot V\subset U$. Observe that for each $n\in D$
$$((x,\eps),(\eps,x))\cdot((b_n,c_n),(d_nx,e_n))=((b_nx,c_n),(d_n,e_n))\notin U.$$
Hence the infinite set $\{((b_n,c_n),(d_nx,e_n)):n\in D\}$ is contained in $U\setminus V$ which implies a contradiction.

$(i=4)$ In this case for each $n\in D$
$$((b_n,c_n),(d_n,e_nx))\in U\cap D_{a_n}\quad \hbox{ and }\quad ((b_n,c_nx),(d_n,e_n))\notin U\cap D_{a_n}.$$
Consider the product $0\cdot ((\eps,x),(x,\eps))=0$. There exists an open neighborhood $V\subset U$ of $0$ such that $V\cdot ((\eps,x),(x,\eps))\subset U$. Observe that for each $n\in D$
$$((b_n,c_n),(d_n,e_nx))\cdot((\eps,x),(x,\eps))=((b_n,c_nx),(d_n,e_n))\notin U.$$
Hence the infinite set $\{((b_n,c_n),(d_n,e_nx)):n\in D\}$ is contained in $U\setminus V$ which implies a contradiction.

The obtained contradiction completes the proof of the proposition.
\end{proof}

\begin{theorem}\label{lc}
A locally compact semitopological semigroup $\mathcal{M}_1$ is either compact or discrete.
\end{theorem}
\begin{proof}
If we assume that the free monoid $F_1$ is taken over the singleton $\{a\}$, then it is easy to see that non-zero $\mathscr D$-classes in $\mathcal{M}_1$ have the form $D_{a^n}$, $n\in \mathbb{N}$.
Assuming that $\mathcal{M}_1$ is not discrete, fix any compact infinite open neighborhood $U$ of $0$. Let $A=\{n\in \mathbb{N}: D_{a^n}\cap U\neq\emptyset\}$. Since $\mathscr{D}$-classes are finite (see Lemma~\ref{l2}), the set $A$ is infinite. To derive a contradiction, assume that the set $B=\mathbb{N}\setminus A$ is infinite as well. In this case there exists an infinite subset $\{n_k:k\in\mathbb{N}\}\subset A$ such that $\{n_k+1:k\in\mathbb{N}\}\subset B$.
Consider the product $((a,a),(\eps,\eps))\cdot 0=0$. The continuity of left shifts in $\mathcal{M}_1$ implies that there exists an open neighborhood $V\subset U$ of $0$ such that
$((a,a),(\eps,\eps))\cdot V\subset U$. Proposition~\ref{pro} ensures that $D_{a^{n_k}}\subset U$ for all but finitely many $k\in \mathbb{N}$. Since the set $U\setminus V$ is finite, there exists $k\in \mathbb{N}$ such that $D_{a^{n_k}}\subset V$, in particular, $((\eps,\eps),(a^{n_k},a^{n_k}))\in V$. It follows that
$$((a,a),(\eps,\eps))\cdot((\eps,\eps),(a^{n_k},a^{n_k}))=((a,a),(a^{n_k},a^{n_k}))\in D_{a^{n_k+1}}\subset \mathcal{M}_1\setminus U.$$
The obtained contradiction implies that the set $B=\mathbb{N}\setminus A$ is finite. By Proposition~\ref{pro}, all but finitely many
$\mathscr D$-classes are contained in $U$. Lemma~\ref{l2} implies that the set $\mathcal{M}_1\setminus U$ is finite. Hence the semitopological semigroup $\mathcal{M}_1$ is compact.
\end{proof}

Since the free inverse semigroup over a singleton with adjoined zero is isomorphic to the semigroup $\mathcal{M}_1$, Theorem~\ref{lc} implies the following:

\begin{corollary}
A locally compact semitopological free inverse semigroup over a singleton with adjoined zero is either compact or discrete.
\end{corollary}

\begin{lemma}\label{ls}
Each two-sided ideal $I\subset F_{\lambda}$ generates a locally compact inverse semigroup topology $\tau_I$ on $\mathcal{M}_{\lambda}$.
\end{lemma}

\begin{proof}
We define $\tau_I$ as follows: each non-zero element of $(\mathcal{M}_{\lambda},\tau_I)$ is isolated and an open neighborhood base at $0$ consists of the sets $U_A=\cup_{a\in A}{D_a}\cup\{0\}$, where $A\subset I$, $|I\setminus A|<\omega$. Since $\mathscr D$-classes are finite in $\mathcal{M}_{\lambda}$ it is easy to see that the space $(\mathcal{M}_{\lambda},\tau_I)$ is locally compact and Hausdorff.  Note that $U_A^{-1}=U_A$ for each $A\subset I$. Hence the inversion is continuous in $(\mathcal{M}_{\lambda},\tau_I)$. Fix any $A\subset I$ such that the set $B=I\setminus A$ is finite. Enumerate $B=\{b_1,\ldots, b_n\}$ and let
$$C=\{x\in F_{\lambda}: x \hbox{ is a subword of }b_i\hbox{ for some }i\leq n\}\setminus\{\eps\}.$$
Note that the set $C$ is finite and put $J=I\setminus C$. One can easily check that $J$ is a two-sided ideal in $F_{\lambda}$. We claim that $U_J$ is a two sided ideal in $\mathcal{M}_{\lambda}$. Indeed, fix any $((a,b),(c,d))\in\mathcal{M}_{\lambda}$ and $((x,y),(u,v))\in D_p\subset U_J$. Note that $p\in J$. Proposition~\ref{p1} implies that either ${\bf z}=((a,b),(c,d))\cdot ((x,y),(u,v))=0\in U_J$ or there exist words $s,r\in F_{\lambda}$ such that ${\bf z}\in D_{spr}$. Since $J$ is a two-sided ideal in $F_{\lambda}$, the word $spr$ belongs to $J$ which implies that ${\bf z}\in D_{spr}\subset U_J$. Hence $U_J$ is a left ideal in $\mathcal{M}_{\lambda}$. Similarly it can be shown that $U_J$ is a right ideal. Thus, for each $((a,b),(c,d))\in\mathcal{M}_{\lambda}$ we have that $$((a,b),(c,d))\cdot U_J\cup U_J\cdot ((a,b),(c,d))\cup U_J\cdot U_J\subset U_J\subset U_A.$$ It follows that $(\mathcal{M}_{\lambda},\tau_I)$ is a topological inverse semigroup.
\end{proof}

A family $\mathcal{A}$ of infinite subsets of $\mathbb{N}$ is called {\em almost disjoint} if $|A\cap B|<\omega$ for each $A,B\in \mathcal{A}$.
By Theorem 1.3 from~\cite{Kun}, there exists an almost disjoint family $\mathcal{A}\subset \mathcal{P}(\mathbb{N})$ of cardinality continuum.
The next proposition shows that Theorem~\ref{lc} cannot be generalized for the McAlister semigroup $\mathcal{M}_{2}$.

\begin{proposition}
The semigroup $\mathcal{M}_2$ admits continuum many distinct locally compact inverse semigroup topologies.
\end{proposition}

\begin{proof}
Assume that the free monoid $F_{2}$ is taken over the set $\{a,b\}$ and consider the ideals $I_n=F_2 ab^na F_{2}\subset F_2$, $n\in\mathbb{N}$. Note that $ab^na\in I_m$ iff $n=m$ for each $n,m\in\mathbb{N}$. Let $\mathcal{A}$ be any almost disjoint family of infinite subsets of $\mathbb{N}$ of cardinality continuum. For each $A\in \mathcal{A}$ let $J_A=\cup_{n\in A}I_n$ and note that $J_A$ is an ideal in $F_2$. We claim that the set $J_A\setminus J_B$ is infinite for each distinct $A,B\in \mathcal{A}$. Indeed, consider the infinite subset $C=\{ab^na:n\in A\}\subset J_A$. Recall that $ab^na\in J_B$ iff $n\in B$. Since the sets $A$ and $B$ are almost disjoint the set $C\setminus J_B\subset J_A\setminus J_B$ is infinite.
Therefore, for each distinct $A,B\in \mathcal{A}$ the locally compact topologies $\tau_{J_A}$ and $\tau_{J_B}$ (see Lemma~\ref{ls}) are distinct. Recall that by Proposition~\ref{pr}, there exist only four different automorphisms of $\mathcal{M}_2$. The Pigeonhole principle implies that we can find a subfamily $\{\tau_{\alpha}:\alpha\in\mathfrak{c}\}\subset \{\tau_{J_A}:A\in\mathcal{A}\}$ such that for each $\alpha\neq \beta$ the locally compact topological inverse semigroups $(\mathcal{M}_2,\tau_{\alpha})$ and $(\mathcal{M}_2,\tau_{\beta})$ are not topologically isomorphic.
\end{proof}

The next lemma provides a method of constructing locally compact inverse semigroup topologies on $\mathcal{M}_{\lambda}$, $\lambda\geq \omega$. Also, it shows that Proposition~\ref{pro} does not hold in the case $\lambda\geq \omega$.

\begin{lemma}
For any infinite cardinal $\lambda$ there exists a locally compact inverse semigroup topology on $\mathcal{M}_{\lambda}$ such that $0$ has an infinite compact neighborhood $U$ intersecting each $\mathscr D$-class by at most one element.
\end{lemma}

\begin{proof}
Assume that the free monoid $F_{\lambda}$ is taken over the set $\{a_{\alpha}:\alpha\in\lambda\}$. For each $n\in \omega\subset \lambda$ put $U_n=\{((a_k,a_k),(a_k,a_k)):n<k<\omega\}\cup\{0\}$. Let $\tau$ be the topology on $\mathcal{M}_{\lambda}$ defined as follows: each non-zero element of $\mathcal{M}_{\lambda}$ is isolated and the family $\{U_n:n\in\omega\}$ forms an open neighborhood base at $0$. Obviously, the space $(\mathcal{M}_{\lambda},\tau)$ is Hausdorff and locally compact ($U_0$ is a compact neighborhood of $0$). Let us check that $\tau$ is an inverse semigroup topology. Since $U_n^{-1}=U_n$ the inversion is continuous. To prove the continuity of the multiplication in $(\mathcal{M}_{\lambda},\tau)$ it is sufficient to consider the following cases:
\begin{enumerate}
\item $0\cdot 0=0$;
\item $x\cdot 0=0$, where $x\in\mathcal{M}_{\lambda}\setminus\{0\}$;
\item $0\cdot x=0$, where $x\in\mathcal{M}_{\lambda}\setminus\{0\}$.
\end{enumerate}

1) Note that $U_n\cdot U_n=U_n$ for each $n\in\omega$. It follows that the semigroup operation is continuous in case 1.

Fix any non-zero element $x=((a,b),(c,d))\in \mathcal{M}_{\lambda}$.

2) If $b\neq \eps$, then let $p$ be the last letter of the word $b\in F_{\lambda}$. Clearly, there exists $n\in \omega$ such that $((p,p),(p,p))\notin U_n$. Then $((a,b),(c,d))\cdot U_n=\{0\}$. If $b=\eps$, then the definition of $\mathcal{M}_{\lambda}$ implies that $d\neq \eps$. Let $p$ be the last letter of the word $d\in F_{\lambda}$. There exists $n\in \omega$ such that $((p,p),(p,p))\notin U_n$. Then $((a,b),(c,d))\cdot U_n=\{0\}$,  witnessing that the semigroup operation is continuous in case 2.

3) If $a\neq \eps$, let $p$ be the last letter of the word $a\in F_{\lambda}$. Clearly, there exists $n\in \omega$ such that $((p,p),(p,p))\notin U_n$. Then $U_n\cdot ((a,b),(c,d))=\{0\}$. If $a=\eps$, then the definition of $\mathcal{M}_{\lambda}$ implies that $c\neq \eps$. Let $p$ be the last letter of the word $c\in F_{\lambda}$. There exists $n\in \omega$ such that $((p,p),(p,p))\notin U_n$. Then $U_n\cdot ((a,b),(c,d))=\{0\}$,  witnessing that the semigroup operation is continuous in case 3.

Hence $\tau$ is an inverse semigroup topology on $\mathcal{M}_{\lambda}$. Lemma~\ref{l1} ensures that for each $n,m\in\omega$,
 $((a_n,a_n),(a_n,a_n))\mathscr D((a_m,a_m),(a_m,a_m))$ iff $n=m$. It follows that $U_0$ is compact neighborhood of $0$ such that  $|U_0\cap D_{u}|\leq 1$ for each $u\in F_{\lambda}\setminus\{\eps\}$.
\end{proof}

\end{document}